\documentclass[11pt]{article}
\usepackage[utf8]{inputenc}
\usepackage{lmodern}
\usepackage{subfiles}
\usepackage{enumitem}
\setenumerate{topsep=0pt,ref={\normalfont(\roman*)},label={\normalfont(\roman*)}, itemsep=0pt} 
\usepackage[indent=15pt,skip=5pt]{parskip}

\usepackage{amsfonts}
\usepackage{amsthm}
\usepackage{amsmath}
\usepackage{amssymb}
\usepackage{amscd}
\usepackage{mathrsfs}
\usepackage{mathtools}
\usepackage{bbm}
\usepackage{esint}

\usepackage{comment}
\usepackage[margin=2.7cm]{geometry}
\usepackage{setspace}
\usepackage{indentfirst}
\usepackage{graphicx}
\usepackage{graphics}
\usepackage{lscape}
\usepackage{pgf,tikz}
\usepackage{tikz-cd}
\usepackage{color}
\usepackage{pict2e}
\usepackage{epic}
\usepackage{epstopdf}
\usepackage{titlesec, titlefoot}
\titleformat{\section}[block]{\Large\bfseries\filcenter}{\thesection}{1em}{}
\usepackage{commath}
\usepackage{float}
\usepackage{caption}
\usepackage{etoolbox}
\usepackage[affil-it]{authblk}
\usepackage{combelow}

\usepackage[hidelinks, bookmarksdepth=3,citecolor=blue,colorlinks]{hyperref}
\hypersetup{bookmarksopen=true} 
\usepackage{hypcap}

\graphicspath{{./Pictures/}}
\allowdisplaybreaks

\expandafter\def\expandafter\normalsize\expandafter{%
\normalsize
\setlength\abovedisplayskip{6pt}
\setlength\belowdisplayskip{6pt}
\setlength\abovedisplayshortskip{6pt}
\setlength\belowdisplayshortskip{6pt}
}

\theoremstyle{plain}

\renewcommand*\thesection{\arabic{section}}
\numberwithin{equation}{section} 

\newtheorem{theorem}{Theorem}[section]
\newtheorem{lemma}[theorem]{Lemma}
\newtheorem*{lemma*}{Lemma}

\newtheorem{corollary}[theorem]{Corollary}

\theoremstyle{definition}
\newtheorem{definition}[theorem]{Definition}

\expandafter\let\expandafter\oldproof\csname\string\proof\endcsname
\let\oldendproof\endproof
\renewenvironment{proof}[1][\proofname]{%
\oldproof[\upshape \bfseries #1]%
}{\oldendproof}

\makeatletter
\def\@makechapterhead#1{%
\vspace*{50\p@}%
{\parindent \z@ \raggedright \normalfont
\interlinepenalty\@M
\Huge\bfseries  \thechapter.\quad #1\par\nobreak
\vskip 40\p@
}}
\makeatother

\newcommand{\eps}{\varepsilon}

\DeclareMathOperator{\ddiv}{div}
\DeclareMathOperator{\supp}{supp}

\DeclareMathOperator \dist{dist}

\DeclareMathOperator{\cof}{cof}

\DeclareMathOperator{\sign}{sign}

\def \R {\mathbb{R}}

\def \N{\mathbb{N}}
\def \D{\textup{D}}

\def \e{\varepsilon}
\def \d{\,\textup{d}}

\def \p{\partial}
\def \mc{\mathcal}
\def \mb{\mathbb}

\def \wstar{\overset{\ast}{\rightharpoonup}}

\def \tp{\textup}
\def \Id{\textup{Id}}
\def \id{\textup{id}}

\begin{document}

	\title{\textbf{Regularity and compactness for critical points\\ of degenerate polyconvex energies}}
	
	\author[1]{{\Large Andr\'e Guerra}}
	\author[2]{{\Large Riccardo Tione}}
		
	\affil[1]{\small Institute for Theoretical Studies,  ETH Zürich,  Clausiusstrasse 47, 8006 Zürich,  Switzerland
	\protect\\
	{\tt{andre.guerra@eth-its.ethz.ch}} \vspace{1em} \ }
	
	\affil[2]{\small Max Planck Institute for Mathematics in the Sciences, Inselstrasse 22, 04103 Leipzig, Germany
	\protect\\
	{\tt{riccardo.tione@mis.mpg.de}}  }
	
	\date{}
	
	\maketitle

	\begin{abstract}
	We study Lipschitz critical points of the energy $\int_\Omega g(\det \D u) \d x$ in two dimensions, where $g$ is a strictly convex function.  We prove that the Jacobian of any Lipschitz critical point is constant, and that the Jacobians of sequences of approximately critical points converge strongly.  The latter result answers in particular an open problem posed by Kirchheim, M\"uller and \v Sver\'ak in 2003. 
	\end{abstract}
	
\unmarkedfntext{
\hspace{-0.75cm}
\emph{Acknowledgments.} AG acknowledges the support of Dr.\ Max R\"ossler, the Walter Haefner Foundation and the ETH Z\"urich Foundation.  We thank Paolo Bonicatto for helpful discussions concerning \cite{Bianchini2016a}.
}
	

\section{Introduction}

Let $\Omega\subset \R^2$ be a bounded domain with $|\partial \Omega|=0$ and consider the energy	
\begin{equation}
\label{eq:energy}
\mb E[u]\equiv \int_\Omega g(\det \D u) \d x, \qquad u\colon \Omega \to \R^2,
\end{equation}
where $g\in C^1(\R)$ is strictly convex.   Energies of this type were studied extensively in the literature,  for instance in connection with elastic fluids or with the  \textit{prescribed Jacobian equation}, see e.g.\ \cite{Carrillo2010,Cellina1995,Dacorogna1981,Dacorogna1995,Evans2006,Kirchheim2003,Mascolo1983,Taheri2002}, as well as the recent works \cite{Tione2022,Yan2023}.   In fact, if $\Omega$ is smooth and $u_0$ is diffeomorphism on $\partial \Omega$,  then one can construct smooth minimizers for \eqref{eq:energy} with boundary data $u_0$ by applying the results of  \cite{Dacorogna1990} to solve the problem
$$
\begin{cases}
\det \D u = \fint_\Omega \det \D u_0(y) \d y,\\
u=u_0 \tp{ on } \partial \Omega.
\end{cases}
$$
For non-$C^1$ data, however,  issues arise with the applicability of this method \cite{GuerraKochLindberg2020c,GuerraKochLindberg2020}, and even if solutions exist they will not be $C^1$ in general.


In this paper we investigate Lipschitz \textit{critical points} of $\mb E$, that is, solutions of
\begin{equation}
\label{eq:EL}
\ddiv(g'(\det \D u) \cof(\D u)^T)=0.
\end{equation}	
The energy $\mb E$ is \textit{polyconvex} \cite{Ball1977} and so in particular it is quasiconvex.  Minimizers of \textit{strongly} quasiconvex energies are smooth almost everywhere \cite{Evans1986}, but for Lipschitz critical points neither quasiconvexity  \cite{Muller2003} nor polyconvexity \cite{Szekelyhidi2004} are enough to imply further regularity.   A challenging open problem is to understand whether further conditions on the solutions (such as stationarity) or further structural assumptions on the energy (such as the ones in this paper) imply regularity. We refer the reader to  \cite{DeLellis2019,Hirsch2023,Hirsch2021a,Kristensen2007,Tione2021} for results in this direction.

Note that $\mb E$ is a \textit{degenerate} polyconvex energy, in the sense that it only controls the Jacobian: in particular,  a Lipschitz critical point $u$ is not $C^1$ in general. However, \textit{if} $u$ is $C^1$, then
$$\det \D u \tp{ is constant in }  \Omega,$$
which is the best type of regularity one can hope for in this problem.  The first question we address is whether this constancy property continues to hold for critical points which are just Lipschitz.  The answer is affirmative:
	
\begin{theorem}[Regularity of exact solutions]\label{thm:reg}
Suppose that $g\in C^1(\R)$ is strictly convex and let $u\in W^{1,\infty}(\Omega,\R^2)$ be a solution of \eqref{eq:EL}. Then $\det \D u$ is constant a.e.\ in $\Omega.$
\end{theorem}	

Previously, in \cite{Tione2022} the second author proved Theorem \ref{thm:reg} under the additional assumption that $\lvert \det \D u\rvert\geq \e>0$ a.e.\ in $\Omega$.  In fact, our proof of Theorem \ref{thm:reg} relies on this result.

As a consequence of Theorem \ref{thm:reg} we find the following perhaps surprising fact:

\begin{corollary}[Critical points are minima]\label{cor:minima}
If $u\in W^{1,\infty}(\Omega,\R^2)$ solves \eqref{eq:EL} then $u$ is a minimizer of $\mb E$, in the sense that
$\mb E[u]\leq \mb E[v]$ for all $v\in u+W^{1,\infty}_0(\Omega,\R^2)$.
\end{corollary}

The second question we want to address concerns Lipschitz maps which solve \eqref{eq:EL} only approximately: is it the case that such maps are close to an actual solution of \eqref{eq:EL}, and if so in which sense? Here we prove the following:

\begin{theorem}[Compactness of approximate solutions]\label{thm:cpt}
Suppose that $g\in C^1(\R)$ is strictly convex.  Let $u_j\wstar u$ in $W^{1,\infty}(\Omega,\R^2)$ and suppose  that there is $(F_j)\subset L^\infty(\Omega,\R^{2\times 2})$ bounded and such that
\begin{equation}
\label{eq:approxEL}
\ddiv ( g'(\det \D u_j) \cof(\D u_j)^T ) = \ddiv( F_j), \qquad F_j \to 0 \text{ in } L^1(\Omega).
\end{equation}
Then  $u$ solves \eqref{eq:EL} and $\det \D u_j \to \det \D u$ in $L^p(\Omega)$ for all $p<\infty$. 
\end{theorem}	

This result is optimal, as in general we do not have $\D u_j\to \D u$ in $L^p(\Omega)$ \cite{Tione2022}, since $\mb E$ is a degenerate energy. 
By rewriting \eqref{eq:EL} as a differential inclusion (cf.\ Section \ref{sec:approx}),  we obtain:
	
\begin{corollary}[Quasiconvexity of the differential inclusion]\label{cor:qc}
The set $K\subset \R^{4\times 2}$, defined by
\begin{equation}
\label{eq:defK}
K_g \equiv \bigg\{\left(\begin{matrix}
A \\ g'(\det A) A
\end{matrix}\right): A \in \R^{2\times 2}\bigg\},
\end{equation}
is quasiconvex. 
\end{corollary}	

In the particular case $g(t)=t^2$, Corollary \ref{cor:qc} answers \cite[Question 10]{Kirchheim2003}.  In that work it was shown that $K_g$ is \textit{rank-one convex}; for subsets of $\R^{4\times 2}$, this is a condition which is strictly weaker than quasiconvexity \cite{Sverak1992a} but which is easier to verify. We refer the reader to \cite{Muller1999a,Rindler2018} for further details on quasiconvexity and rank-one convexity.

\bigskip
The proofs of our main results use various ingredients. First, to show Theorem \ref{thm:reg}, we employ ideas coming from the theory of transport equations, in particular the renormalization properties in the sense of DiPerna--Lions \cite{DiPerna1989} of two-dimensional solutions, which were shown in \cite{Alberti2014,Bianchini2016a}. Next we employ some deep but rather classical results coming from the theory of quasiregular mappings, and we refer the interested reader to \cite{DePhilippis2023,Faraco2008,Kirchheim2008,Sverak1993} for further applications of quasiregular maps in the theory of differential inclusions. To show Theorem \ref{thm:cpt}, the main ingredients are Allard's strong constancy lemma \cite{Allard1986} and the 0-1 law for quasiregular gradient Young measures \cite{Astala2002}. We note that the ideas behind Allard's lemma have recently been applied in \cite{DePhilippis2016} in a rather different problem.





\section{Notation and setup}

We will say that $\Omega\subset\R^m$ is a domain if it is an open, connected bounded set with $|\partial \Omega| = 0$.  Throughout this paper we write $A,B$ for matrices in $\R^{n\times n}$ and we denote by $\langle A,B\rangle$ their Hilbert--Schmidt inner product.  We write $L_A\colon \R^n\to \R^n$ for the linear map induced by $A$. We also write $\cof(A)$ for the cofactor matrix of $A$,  which satisfies
\begin{equation}
\label{eq:cof}
\cof(A) A = A \cof(A) = (\det A) \Id.
\end{equation}
An important property of the cofactor is the so-called Piola identity
\begin{equation}
\label{eq:divcof}
\ddiv(\cof(\D u)^T)=0,
\end{equation}
valid for Lipschitz maps $u\colon \R^n\to \R^n$.  Note that, when $n=2$,  $\cof\colon \R^{2\times 2}\to \R^{2\times 2}$ is \textit{linear}, since $\cof(A) = JA^T J^T$, where $J$ is the matrix corresponding to a rotation by $\pi/2$. 

We say that a function $g\colon \R\to \R$ is \textit{strictly convex} if
$$g(\lambda x + (1-\lambda)y)<\lambda g(x)+(1-\lambda)g(y)$$
whenever $x\neq y$ and $\lambda \in (0,1)$.  If $g \in C^1(\R)$, then it is elementary to see that $g$ is a strictly convex function if and only if $g'$ is strictly increasing.  By \eqref{eq:divcof},  for the results of this paper we can and will always assume that
\begin{equation}\label{eq:g'0}
g'(0)=0.
\end{equation}
When $g$ is strictly convex, i.e.\ when $g'$ is strictly increasing, \eqref{eq:g'0} implies:
\begin{equation}
\label{eq:g'(t)t}
g'(t)t=|g'(t)| |t|\geq 0.
\end{equation}

\section{Exact solutions}\label{sec:exact}


The purpose of this section is to prove Theorem \ref{thm:reg}.  We begin by stating the following deep result, which is essentially due to Bianchini--Gusev, but see also the earlier work \cite{Alberti2014}.

\begin{theorem}[Renormalized solutions]\label{thm:renorm}
Let $u\in W^{1,\infty}$ solve \eqref{eq:EL} and $\beta\in C^\infty_c(\R)$. Then
$$\ddiv(\beta(g'(\det \D u)) \cof(\D u)^T) = 0 .$$
\end{theorem}

\begin{proof}
The result follows directly from \cite[Theorem 6.7]{Bianchini2016a}. Indeed, \eqref{eq:EL} corresponds to a system of two (time-independent) continuity equations $\ddiv(\rho\, b_i)=0$, where the density is $\rho\equiv g'(\det \D u)$ and the bounded vector fields $b_i\colon \Omega \to \R^2$ are the rows of $\cof( \D u)^T$, for $i=1,2$, which are divergence-free by \eqref{eq:divcof}. 
\end{proof}

\begin{proof}[Proof of Theorem \ref{thm:reg}]
We can assume that it is not the case that $\det \D u=0$ a.e.\ in $\Omega$,  since otherwise there is nothing to show.  We may then also assume that $\det \D u>0$ a.e.\ on a set of positive measure, the case $\det \D u<0$ being totally analogous. In particular, there is $\e>0$ small enough so that 
\begin{equation}
\label{eq:measuredet}
\lvert\{x\in \Omega:\det \D u(x)\geq \e\}\rvert>0.
\end{equation}

Let us write $\tilde \e /2 \equiv g'(\e/2)> 0$. We now take $\beta \in C^\infty_c(\R,[0,1])$ such that $\beta(t)=0$ if $t\leq \tilde \e/2$ and $\beta(t)=1$ if $\tilde \e \le t \le 2\|g'(\det(\D u))\|_{L^\infty}$.  
By Theorem \ref{thm:renorm},  the matrix field $\beta(g'(\det \D u))\cof(\D u)^T$ is divergence-free.  Since the statement in Theorem \ref{thm:reg} is local, there is also no loss of generality in assuming that $\Omega$ is simply connected, and so there is $v\in W^{1,\infty}(\Omega,\R^2)$ such that 
\begin{equation}
\label{eq:potential}
\beta(g'(\det \D u)) \D u = \D v.
\end{equation} 
By taking determinants on both sides of this identity, we get
\begin{equation}
\label{eq:detuv}
\beta(g'(\det \D u))^2 \det \D u = \det \D v.
\end{equation} 
The above two identities show that, for a.e.\ $x \in \Omega$, if $\det \D v(x) = 0$, then $\D v(x) = 0$, due to \eqref{eq:g'0} and the definition of $\beta$. On the other hand,  by \eqref{eq:potential},  at a.e.\ $x$ with $\det \D v(x) \neq 0$ we have
\[
0 < \frac{|\D v|^2(x)}{\det \D v(x)} = \frac{|\D u|^2(x)}{\det \D u(x)} \le \frac{2\Lambda^2}{\e},
\]
where $\Lambda > 0$ is the Lipschitz constant of $u$. In other words, there exists $K > 0$ such that $|\D v|^2(x)\leq K \det \D v(x)$ for a.e.\ $x \in \Omega$, i.e.\ the map $v$ is quasiregular. It is well-known that a quasiregular map on a domain is either constant or else its Jacobian is positive a.e., see e.g.\ \cite[Theorem 16.10.1]{Iwaniec2001}. By \eqref{eq:measuredet} we cannot have $\D v =0$ a.e.\ and so $v$ is not constant, hence $\det \D v>0$ a.e.\ in $\Omega$; thus in fact from \eqref{eq:detuv} we deduce $g'(\det\D u) \ge \tilde \e/2$ a.e., which in turn implies $\det \D u \ge \e/2$ a.e.\ in $\Omega$. Finally we conclude by applying  \cite[Theorem 1]{Tione2022}. We note that, in \cite{Tione2022}, the stronger conditions $g\in C^2$ and $g''>0$ are assumed, but that they are not needed for the proof.
\end{proof}

\begin{proof}[Proof of Corollary \ref{cor:minima}]
By Theorem \ref{thm:reg}, there is $c\in \R$ such that $\det \D u = c$ a.e.\ in $\Omega$.  By convexity of $g$, we have a.e.\ the inequality
$$g(\det \D v) \geq g(\det \D u) + g'(c)( \det \D v - \det \D u).$$
The claim now follows by integrating this inequality over $\Omega$,  since $\int_\Omega \det \D v = \int_\Omega \det \D u$ as $\det$ is a null Lagrangian \cite[Theorem 2.3]{Muller1999a} and $u-v\in W^{1,\infty}_0(\Omega)$.
\end{proof}

\section{Approximate solutions}\label{sec:approx}

In this section we prove Theorem \ref{thm:cpt}. In fact, we will reformulate our problem in a slightly different way, using the language of differential inclusions.   Without loss of generality we assume that $\Omega\subset \R^2$ is simply connected, and we can then rewrite \eqref{eq:EL} as
\begin{equation}
\label{eq:potentialmain}
g'(\det \D u) \D u = \D v
\end{equation}
for some $v\in W^{1,\infty}(\Omega,\R^2)$.  This last equation, in turn, is nothing but the differential inclusion
\begin{equation}
\label{eq:diffinc}
\D w\in K_g \quad \tp{a.e.\ in }\Omega,  \qquad w\equiv (u,v)\colon \Omega\to \R^4,
\end{equation}
where $K_g$ is defined in \eqref{eq:defK}. We begin this  section by recalling some well-known but useful results concerning Young measures, and their relation to differential inclusions.

\subsection{Differential inclusions and Young measures}

The results and discussion of this subsection are standard, and the reader can find proofs as well as further information in \cite{Muller1999a,Rindler2018}. We let $\Omega \subset \R^m$ be a domain.


\begin{definition}
A parametrized family of probability measures $\nu=(\nu_x)_{x\in \Omega}$ on $\R^{d}$ is said to be a \textit{Young measure} if there is a weakly-$\ast$ convergent sequence $(z_j)\subset L^{\infty}(\Omega,\R^d)$ which generates $\nu$, in the sense that
$$f(z_j) \wstar \left(x\mapsto \langle\nu_x, f\rangle \right) \quad \text{ in } L^\infty(\Omega),$$
for all $f\in C_0(\R^{d})$. We always write $\langle \nu_x, f\rangle \equiv \int_{\R^{d}} f(A) \d \nu_x(A)$ and $\overline{\nu_x}\equiv \langle\nu_x,\id\rangle$. 

A Young measure $(\nu_x)_{x\in \Omega}$ is \textit{homogeneous} if there is a probability measure $\nu$ in $\R^{d}$ such that $\nu_x=\nu$ for a.e.\ $x$ in $\Omega$; in this case, we naturally identify $(\nu_x)_{x\in \Omega}$ with $\nu$.  

If $d=n\times m$ and the sequence $(z_j)$ satisfies $z_j = \D w_j$ for some $(w_j)\subset W^{1,\infty}(\Omega,\R^n)$ then we say that $\nu$ is a \textit{gradient Young measure}.
\end{definition}

The set of homogeneous gradient Young measures supported on $K\subset \R^{n\times m}$ is denoted by $\mc M^\tp{qc}(K)$; whenever $K=\R^{n\times m}$ we suppress it from the notation.  The justification for the choice of the superscript ``qc'' will become clear at the end of the next theorem,  which collects some fundamental properties of Young measures.

\begin{theorem}\label{thm:props}
Young measures have the following properties:
\begin{enumerate}
\item \label{it:cpt}  {\normalfont Compactness:} if $(z_j)\subset L^{\infty}(\Omega,\R^d)$ is bounded then,  up to a non-relabeled subsequence, $(z_j)$ generates a Young measure.
\item \label{it:strong} {\normalfont Strong convergence:} assume   $(z_j)$ generates a Young measure $(\nu_x)_{x\in\Omega}$ and $z_j\wstar z$ in $L^\infty(\Omega,\R^d)$.  Then $z_j\to z$ in $L^p$ for all $p<\infty$ if and only if $\nu_x=\delta_{z(x)}$ for a.e. $x\in \Omega$.
\item \label{it:support} {\normalfont Support:} assume $(z_j)$ generates a Young measure $(\nu_x)_{x\in \Omega}$. For a compact subset $K\subset \R^d$ we have $\dist(z_j,K)\to 0$ in $L^1(\Omega)$ if and only if  $\supp \nu_x\subseteq K$ for a.e.\ $x\in \Omega$.\end{enumerate}
Gradient Young measures have the following additional properties:
\begin{enumerate}[resume]
\item \label{it:bc} {\normalfont Boundary conditions:} if $(\nu_x)_{x\in \Omega}$ is a gradient Young measure then it is generated by a sequence $(\D w_j)$ with $\supp(w_j-w)\Subset \Omega$, where $\D w(x)\equiv \overline{\nu_x}$.
\item \label{it:loc}  {\normalfont Localization:} if $(\nu_x)_{x\in \Omega}$ is a gradient Young measure then $\nu_x\in \mc M^\tp{qc}$ for a.e.\ $x\in \Omega$.
\item \label{it:duality}{\normalfont Duality with quasiconvexity:} a probability measure $\nu$ supported on a compact subset $K\subset \R^{n\times m}$ is in $\mc M^\tp{qc}(K)$ if and only if $f(\overline \nu) \leq \langle f, \nu\rangle$ for all quasiconvex integrands $f$.
\end{enumerate}
\end{theorem}

We recall that an integrand $f\colon \R^{n\times m}\to \R$ is said to be \textit{quasiconvex} if
\begin{equation}
\label{eq:qc}
f(A)\leq \fint_\Omega f(A+\D \varphi) \d x \quad \text{for all } A\in \R^{n\times m}, \text{ all } \varphi \in C^\infty_c(\Omega,\R^n).
\end{equation}
A typical example of quasiconvex integrands are the minors, which in fact satisfy \eqref{eq:qc} with equality. This implies that, for any $\nu \in \mc M^\tp{qc}$ and any minor $M$, we have
\begin{equation}
\label{eq:minors}
M(\overline \nu) = \langle \nu, M\rangle.
\end{equation}
For a compact set $K\subset \R^{n\times m}$ we can define its \textit{quasiconvex hull} $K^\tp{qc}$ as the set of points which cannot be separated from $K$ by a quasiconvex function; a set is said to be quasiconvex if $K=K^\tp{qc}$.  By \cite[Theorem 4.10]{Muller1999a},  the quasiconvex hull coincides with 
\begin{equation}
\label{eq:Kqc}
K^\tp{qc}= \{\overline \nu: \nu \in \mc M^\tp{qc}(K)\}.
\end{equation}
Whenever $K$ is non-compact,  we set $K^\tp{qc}\equiv \bigcup_{r>0} [K\cap \overline{B_r(0)}]^\tp{qc}$.

%
%

\subsection{Stability and compactness of approximate solutions}

We now return to the case $m=4$, $n=2$ and to the set  $K_g$ defined in \eqref{eq:defK}. We begin with some technical lemmas:

\begin{lemma}\label{lem:dist}
Let $\Omega \subset \R^2$ be a simply connected domain. Let $(u_j)\subset W^{1,\infty}(\Omega,\R^2)$ be a bounded sequence and $g\in C^1(\R)$. The following are equivalent:
\begin{enumerate}
\item\label{it:wj} there is $(v_j)\subset W^{1,\infty}(\Omega,\R^2)$ bounded with $\dist(\D w_j,K_g) \to 0$ in $L^1$,  where $w_j= (u_j,v_j)$;
\item\label{it:vj} there is $(v_j)\subset W^{1,\infty}(\Omega,\R^2)$ bounded such that $g'(\det \D u_j) \D u_j- \D v_j\to 0$ in $L^1$;
\item\label{it:fj}  \eqref{eq:approxEL} holds for a bounded sequence $(F_j)\subset L^\infty(\Omega,\R^{2\times 2})$.
\end{enumerate}
\end{lemma}

\begin{proof}
For a bounded sequence $(w_j)\subset W^{1,\infty}(\Omega,\R^4)$,  one has $\dist(\D w_j, K_g)\to 0$ in $L^1(\Omega)$  if and only if $f(\D w_j)\to 0$ in $L^1(\Omega)$, where $f\colon\R^{4\times 2}\to \R$ is any continuous function such that $\{f=0\}=K_g$, see e.g.\ \cite{Muller1999a} for a similar argument.   In particular,  taking
$$f\colon \left(\begin{matrix}
A \\ B
\end{matrix}\right)\mapsto \lvert g'(\det A) A-B\rvert,$$
we obtain the equivalence between  \ref{it:wj} and  \ref{it:vj}. Note also that, 
by linearity of the cofactor,  \ref{it:vj} is equivalent to
$$ g'(\det \D u_j) \cof(\D u_j)^T- \cof(\D v_j)^T \to 0 \text{ in } L^1(\Omega).$$
In particular, if this convergence holds, then by \eqref{eq:divcof} we may take $F_j$ to be the quantity above, proving that \ref{it:vj}$\implies$\ref{it:fj}.  The converse follows from the fact that, if \eqref{eq:approxEL} holds, then 
$$g'(\det \D u_j) \cof(\D u_j)^T- F_j $$ 
is a bounded divergence-free matrix field, which can therefore be written as $\cof(\D v_j)^T$ for some Lipschitz $v_j\colon \Omega\to \R^2$, since $\Omega$ is simply connected. Then the convergence in \ref{it:vj} holds since $F_j\to 0$ in $L^1(\Omega)$.
\end{proof}

\begin{lemma}\label{lem:topolo}
Consider domains $\Omega_2 \Subset \Omega_1 \Subset \Omega$.  Let $A\in \R^{2\times 2}$ be non-singular and assume that
\begin{equation}
\label{eq:hyppre}
u_j \to  L_A \text{ in } C^0(\overline \Omega), \qquad u_j = L_A \text{ on } \p \Omega.
\end{equation}
There exists $j_0 = j_0(\Omega_1,\Omega_2) > 0$ such that following hold:
\begin{enumerate}
\item $u_j(\Omega_2)\Subset L_A(\Omega_1)\Subset L_A(\Omega)$, for all $j \ge j_0$;\label{pro:fi}
\item for all $y\in u_j(\Omega_2)$ we have $u_j^{-1}(y)\subseteq \Omega_1$, for all $j \ge j_0$;\label{pro:se}
\item $\lim_{j}|u_j(\Omega_2)| = |L_A(\Omega_2)|$.\label{pro:thi}
\end{enumerate}
\end{lemma}
\begin{proof}
We will just prove \ref{pro:thi},  as \ref{pro:fi} and \ref{pro:se} are much simpler.  
To prove \ref{pro:thi},  note that by uniform convergence, for all $\e > 0$, there exists $j_1$ such that
\begin{equation}\label{eq:epsn}
u_j(\Omega_2) \subset B_\eps(L_A(\Omega_2)),\quad \forall j \ge j_1,
\end{equation}
where $B_\eps(E)$ denotes the $\eps$ neighborhood of the set $E \subset \R^2$. 
In turn, \eqref{eq:epsn} implies
\begin{equation}\label{eq:limsup}
\limsup_{j\to \infty }|u_j(\Omega_2)| \le |L_A(\Omega_2)|,
\end{equation}
since $|\partial \Omega_2| = 0$. To show the opposite inequality, consider any domain $\Omega_3 \Subset \Omega_2$. We claim that there exists $j_2 = j_2(\Omega_3) > 0$ such that if $j \ge j_2$ then
 \begin{equation}\label{eq:epsnb}
L_A(\Omega_3) \subseteq u_j(\Omega_2).
\end{equation}
To see this, we start by noticing that, for all $j \in \N$,
 \begin{equation}\label{eq:weaker}
L_A(\Omega_3) \subseteq u_j(\Omega).
\end{equation}
This can be seen using the properties of the topological degree, see e.g.\ \cite[Theorems 2.1 and 2.4]{Fonseca1995}, exploiting the fact that $u_j = L_A$ on $\partial \Omega$ for all $j$. Suppose now that \eqref{eq:epsnb} is false, so that, up to non-relabeled subsequences, we find $y_j \in L_A(\Omega_3)$ with $y_j \notin u_j(\Omega_2)$ for all $j$. We can additionally suppose that $y_j \to \bar y \in L_A(\overline \Omega_3)$. Due to \eqref{eq:weaker}, we see that there exists a sequence $(x_j) \subset \Omega$ such that $u_j(x_j) = y_j$. Our contradiction assumption implies $x_j \notin \Omega_2$ for all $j$. In particular, up to a further subsequence, we see that $x_j \to \bar x \in \overline{\Omega\setminus\Omega_2} \subset \overline{\Omega}\setminus \Omega_2.$ As $\overline{\Omega_3} \subset \Omega_2$ and $A$ is nonsingular, the uniform convergence of $u_j\to L_A$ readily yields a contradiction. Thus, \eqref{eq:epsnb} holds, and we deduce in turn that
\[
|L_A(\Omega_3)| \le \liminf_j|u_j(\Omega_2)|.
\]
Due to the arbitrarity of $\Omega_3$,  combining the latter with  \eqref{eq:limsup} we conclude the validity of \ref{pro:thi}.
\end{proof}

We can now show the main result of this section:

\begin{theorem}\label{thm:mainapprox}
Let $\Omega \subset \R^2$ be a simply connected domain and let $(u_j),(v_j)$ be sequences bounded in $W^{1,\infty}(\Omega,\R^2)$. Let $A,B\in \R^{2\times 2}$ and assume the following conditions:
\begin{align}
\label{eq:hyp}
\begin{split}
&u_j = L_A \text{ and } v_j = L_B \text{ on } \p \Omega, \\
&u_j \wstar L_A \text{ and } v_j \wstar L_B \text{ in } W^{1,\infty},\\
&(\D u_j,\D v_j) \text{ generates } \nu \in \mc M^\tp{qc}(K_g),\\
&g'(\det \D u_j)\D u_j - \D v_j \to 0 \text{ in }L^1.
\end{split}
\end{align}
Then  we have
$$\supp \nu\subseteq \bigg\{\left(\begin{matrix}
M\\N
\end{matrix}\right)\in K_g: \det M = \det A\bigg\}$$
and, in particular,   $\det \D u_j \to \det A \tp{ in }L^1$ and $(A,B)=\overline \nu\in  K_g$.
\end{theorem}

\begin{proof}
Note that \eqref{eq:hyp} implies that $(u_j,v_j)\to (L_A,L_B)$ in $C^0(\Omega)$,  so Lemma \ref{lem:topolo} is applicable.  We now define some notation.  Consider the  projections $\pi_{12} , \pi_{34}\colon \R^{4\times 2}\to \R^{2\times 2}$ defined by
$$\pi_{12} \colon \left(\begin{matrix}
A \\ B 
\end{matrix}\right)\mapsto A,\qquad 
\pi_{34} \colon \left(\begin{matrix}
A \\ B 
\end{matrix}\right)\mapsto B.
$$
Let $\nu_{12},\nu_{34}$ be the homogeneous gradient Young measures generated respectively by  $(\D u_j)$ and $(\D v_j)$; we have the pushforward relations $\nu_{12} = (\pi_{12})_\#\nu$ and $\nu_{34}=(\pi_{34})_\# \nu$. 
We split the rest of the proof into several steps.

\textbf{Step 1.} Let $(u_j),(v_j)$ be sequences satisfying the hypotheses in \eqref{eq:hyp}. By Lemma \ref{lem:dist} we see that \eqref{eq:approxEL} holds.  
Given $M\in \R^{2\times 2}$,  testing \eqref{eq:approxEL} against $\varphi_j(x)=M(u_j(x)-Ax)$ and using \eqref{eq:cof}, we see that
\begin{align*}
o(1) = \int_\Omega \langle \D \varphi_j,F_j \rangle & = \int_\Omega g'(\det \D u_j) \langle\D \varphi_j, \cof(\D u_j)^T\rangle \d x\\
& = \int_\Omega g'(\det \D u_j) \left[\det \D u_j \langle M, \Id \rangle - \langle M A, \cof(\D u_j)^T\rangle \right]\d x,
\end{align*}
where $o(1)$ denotes a quantity which vanishes as $j\to \infty$.  Equivalently, 
$$\Big\langle M, \Big(\int_\Omega g'(\det \D u_j) \det \D u_j \d x \Big)\Id \Big\rangle = \Big \langle M, \Big(\int_\Omega g'(\det \D u_j) \cof(\D u_j)^T \d x \Big) A^T\Big\rangle +o(1)$$
and, since $M$ is arbitrary,  we see that
\begin{align}
\label{eq:Bj}
\begin{split}
\lim_{j\to \infty} \Big(\int_\Omega g'(\det \D u_j) \det \D u_j \d x \Big) \Id & =\lim_{j\to \infty} \Big( \int_\Omega g'(\det \D u_j) \cof(\D u_j)^T \d x\Big) A^T\\
& = \lim_{j\to \infty} \Big(\int_\Omega \cof(\D v_j)^T \d x\Big) A^T\\
& = \cof(B)^T A^T,
\end{split}
\end{align}
where in the second line we used the convergence hypothesis in \eqref{eq:hyp} and in the last line we used the linearity of $\cof$ and the  boundary conditions on $v_j$. Similarly, we also have
\begin{equation}
\label{eq:detB}
\det B  = \fint_\Omega \det \D v_j  \d x = \lim_{j\to \infty} \fint_\Omega g'(\det \D u_j)^2 \det \D u_j \d x ,
\end{equation}
the first equality following because $\det$ is a null Lagrangian and the second by the last hypothesis in \eqref{eq:hyp}.

\textbf{Step 2.} In this step we deal with the case where $\det A =0$.  By taking the determinant in \eqref{eq:Bj} we get
\begin{equation*}
\int_{\R^{2\times 2}} g'(\det M) \det(M) \d \nu_{12}(M)  = \lim_{j\to \infty} \fint_\Omega g'(\det \D u_j) \det \D u_j \d x = 0,
\end{equation*}
thus $\det(M)=0$ for $\nu_{12}$-a.e.\ $M$, as $t\mapsto g'(t)t$ is non-negative and vanishes only at zero, according to \eqref{eq:g'(t)t}. Hence the desired claim on the support of $\nu_{12}$ holds and the latter claims follow exactly as in the case where $\det A\neq 0$,  which we will detail in Step 10 below.

\textbf{Step 3.}
For the rest of the proof we assume that $\det A\neq 0$, so that $L_A(\Omega)$ is a domain. 
 Let $\eta\in C^\infty_c(\R^2,\R^2)$ be such that $\eta|_{L_A(\partial \Omega)}=0$.  Thus $\eta\circ u_j\in W^{1,\infty}_0(\Omega,\R^2)$ and so, testing  \eqref{eq:approxEL} against this map and using \eqref{eq:cof} and \eqref{eq:g'(t)t}, we get
\begin{align}
\label{eq:areaform}
\begin{split}
 \int_\Omega \langle \D(\eta\circ u_j), F_j\rangle \d x & = \int_\Omega g'(\det \D u_j) \langle \D\eta \circ u_j \,\D u_j, \cof(\D u_j)^T\rangle \d x \\
&  = \int_\Omega g'(\det \D u_j) \ddiv \eta \circ u_j \det \D u_j \d x \\
& = \int_\Omega |g'(\det \D u_j)| (\ddiv \eta \circ u_j ) \lvert\det \D u_j\rvert \d x \\
& = \int_{\R^2} 1_{u_j(\Omega)}(y) S_{u_j}(y) \ddiv \eta(y) \d y,
\end{split}
\end{align}
where we used the area formula (cf.\ \cite[Corollary 8.11]{Maggi2012}) and we defined, for a.e.\ $y\in u_j(\Omega)$, 
\begin{equation}
\label{eq:defSj}
S_{u_j}(y) \equiv \sum_{x\in u_j^{-1}(y)} |g'(\det \D u_j(x))|.
\end{equation}
Let us note in passing that, since $(\D u_j)$ is bounded in $L^{\infty}$, we have
\begin{equation}
\label{eq:bddSj}
\sup_{j} \int_{\R^2} 1_{u_j(\Omega)}(y) S_{u_j}(y) \d y = \sup_j \int_\Omega |g'(\det \D u_j(x))| \lvert \det \D u_j(x)\rvert \d x \leq C.
\end{equation}

\textbf{Step 4.}
To deal with the left-hand side of \eqref{eq:areaform},  we consider the push-forward measure $m_j\equiv (u_j)_\# (F_j\D u_j^T \d x) $, which we can write as $m_j=P_j \alpha_j$ for a finite, positive measure $\alpha_j$ and for an $\alpha_j$-measurable matrix field $P_j\colon \R^2 \to \R^{2\times 2}$ with $|P_j(x)| =1$ for $\alpha_j$-a.e.\ $x \in \Omega$. Then, according to the general change of variables formula \cite[Proposition 2.14]{Maggi2012}, we have
\begin{equation}\label{eq:preciserepre}
 \int_\Omega \langle \D(\eta\circ u_j), F_j\rangle \d x = \int_\Omega \langle \D\eta\circ u_j, F_j\D u_j^T\rangle \d x = \int_{\R^2} \langle \D\eta(y), P_j(y)\rangle \d \alpha_j(y)
\end{equation}
for all $\eta\in C^\infty_c(\R^2,\R^2)$.
We also notice that,  by \eqref{eq:approxEL}, we have
\begin{equation}\label{eq:precisemea}
 \|m_j\|_{\mc M} = \alpha_j(\R^2) \le \|F_j\D u_j^T\|_{1} \le \|F_j\|_1\|\D u_j\|_\infty \to 0.
\end{equation}
Thus, combining \eqref{eq:areaform} and \eqref{eq:preciserepre}, we have shown that,  for $\eta\in C^\infty_c(\R^2)$ such that $\eta|_{L_A(\partial \Omega)}=0$,
\begin{equation}
\label{eq:areaconc}
\int_{\R^2} 1_{u_j(\Omega)} S_{u_j} \ddiv \eta \d y  = \int_{\R^2} \langle \D \eta, P_j\rangle \d \alpha_j.
\end{equation}


\textbf{Step 5.} By taking $f_j\equiv  1_{u_j(\Omega)} S_{u_j}\geq 0$,  note that \eqref{eq:areaconc} shows that
\begin{equation}
\label{eq:divfj}
\D f_j = \ddiv (P_j \alpha_j)\quad \text{in } \mathscr D'(L_A(\Omega)).
\end{equation}
At this point we will employ Allard's strong constancy lemma \cite{Allard1986}, which we state here in a somewhat simpler form:

\begin{lemma}[Strong constancy lemma]\label{lem:allard}
Let $(f_j)\subset L^1(\tilde \Omega)$ be a bounded sequence such that 
$$f_j\geq 0,\qquad \D f_j = \ddiv G_j \text{ for some sequence  }\lvert|G_j \rvert|_{\mc M(\tilde \Omega)}\to 0.$$ Then there is a constant $c\geq 0$ such that $f_j\to c$ in $L^1_\tp{loc}(\tilde \Omega)$.
\end{lemma}

Take arbitrary domains $\Omega_2\Subset \Omega_1\Subset \Omega$. Due to Lemma \ref{lem:topolo}\ref{pro:fi}, for $j$ large enough we have $u_j(\Omega_2)\Subset L_A(\Omega_1)\Subset L_A(\Omega)$.  Thus, by Lemma \ref{lem:allard}, which is applicable by \eqref{eq:defSj},  \eqref{eq:bddSj}, \eqref{eq:precisemea} and \eqref{eq:divfj},  we see that
\begin{equation}
\label{eq:Stoc}
S_{u_j} \to c \text{ in }L^1(L_A(\Omega_1)).
\end{equation}

\textbf{Step 6.} Fix $\e>0$ and consider the set $E_j\equiv \Omega_2\cap \{|g'(\det \D u_j)|\geq c+\e\}$.  Again by Lemma \ref{lem:topolo}\ref{pro:fi} we have $u_j(E_j)\subset L_A(\Omega_1)$,  so using \eqref{eq:Stoc} and  applying the area formula twice, we get
\begin{align*}
c \int_{E_j} \lvert \det \D u_j\rvert \d x \geq c |u_j(E_j)| & = \int_{u_j(E_j)} S_{u_j}(y) \d y + o(1)\\
& \ge \int_{E_j} |g'(\det \D u_j)| \lvert \det \D u_j\rvert \d x +o(1)\\
& \geq (c+\e)\int_{E_j} \lvert \det \D u_j\rvert \d x+o(1).
\end{align*}
This implies that
$$\lim_{j\to \infty} \int_{E_j} \lvert \det \D u_j(x)\rvert \d x=0 .$$
Since $g'(0)=0$ and $g'$ is injective, there is $\tilde \e>0$ such that $\lvert\det \D u_j\rvert\geq \tilde \e$ a.e.\ in $E_j$, and it follows that $\lim_{j\to \infty}|E_j|\to 0$. Since $\Omega_2\Subset \Omega$ is arbitrary, by taking a compact exhaustion of $\Omega$ we deduce that 
\begin{equation}
\label{eq:Ejto0}
\lim_{j\to \infty} |\{|g'(\det \D u_j)|\geq c+\e\}|= 0.
\end{equation}

\textbf{Step 7.}  Recall from Lemma \ref{lem:topolo}\ref{pro:se} that for all $j$ large enough and all $y\in u_j(\Omega_2)\Subset L_A(\Omega_1)$ we have $u_j^{-1}(y)\Subset \Omega_1$. Hence, applying the area formula twice, for such $j$ we have
\begin{align*}
\int_{\Omega_1} |g'(\det \D u_j(x))| |\tp{det}\, \D u_j(x)| \d x &= \int_{u_j(\Omega_1)} \sum_{x \in u_j^{-1}(y)\cap\Omega_1}|g'(\det \D u_j(x))| \d y\\
&\geq  \int_{u_j(\Omega_2)}\sum_{x \in u_j^{-1}(y)\cap\Omega_1}|g'(\det \D u_j(x))| \d y\\
& =  \int_{u_j(\Omega_2)}\sum_{x \in u_j^{-1}(y)}|g'(\det \D u_j(x))| \d y\\
&= \int_{u_j(\Omega_2)} S_{u_j}(y) \d y\geq \int_{\Omega_2} |g'(\det \D u_j(x))| |\tp{det}\, \D u_j(x)| \d x.
\end{align*}
Thus, since $\nu_{12}$ is a homogeneous gradient Young measure,  using  \eqref{eq:g'(t)t} we get
\begin{align*}
|\Omega_1| \int_{\R^{2\times 2}} g'(\det M) \det M \d \nu_{12}(M) \geq \lim_{j\to \infty} \int_{u_j(\Omega_2)} S_{u_j} \d y \geq |\Omega_2| \int_{\R^{2\times 2}} g'(\det M) \det M \d \nu_{12}(M).
\end{align*}
By \eqref{eq:Stoc} and Lemma \ref{lem:topolo}\ref{pro:fi},\ref{pro:thi} we must also have
$$\lim_{j\to \infty} \int_{u_j(\Omega_2)} S_{u_j}(y) \d y = \lim_{j\to \infty} \int_{u_j(\Omega_2)} (S_{u_j}(y) -c )  \d y + c\lim_j|u_j(\Omega_2)| = c \lvert \det A\rvert |\Omega_2|.$$
%
Since $\Omega_2\Subset \Omega_1\Subset \Omega$ are arbitrary domains, combining the last two displays yields
\begin{equation}
\label{eq:c1j}
c \lvert\det A\rvert = \int_{\R^{2\times 2}} g'(\det M) \det M \d \nu_{12}(M)
\end{equation}
If we take determinants in \eqref{eq:Bj} and use this identity,  since $\det A \neq 0$, we also see that
\begin{equation}
\label{eq:c2j}
c^2 \det A	= \det B.
\end{equation}

\textbf{Step 8.}
%
%
Note that \eqref{eq:Ejto0} shows that $\nu_{12}$ is supported in $\{M\in \R^{2\times 2}:|g'(\det M)|\leq c\}$. Furthermore, recall \eqref{eq:g'(t)t}. These facts, combined with \eqref{eq:detB} and \eqref{eq:c1j}, yield
\begin{align*}
\lvert\det B \rvert &= \Big|\int_{\R^{2\times 2}} g'(\det M) [g'(\det M) \det(M)]\d \nu_{12}(M)\Big| \\
& \leq c \int_{\R^{2\times 2}} g'(\det M) \det(M) \d \nu_{12}(M)= c^2 \lvert \det A\rvert.
\end{align*}
By \eqref{eq:c2j} we see that we must have equality in the inequality above. It follows that 
$$g'(\det M)= \sign(\det A) c \qquad \text{ for } \nu_{12}\text{-a.e.\ } M\text{ with }g'(\det M) \det M \neq 0.$$  As in Step 2, $g'(\det M) \det M=0$ if and only if $\det M=0$. Note that $\supp \nu_{12}\not \subseteq\{M:\det M=0\}$, for otherwise by \eqref{eq:minors} we would have $\det A = \det \overline{\nu_{12}} = 0$, in contradiction to our assumption on $A$. Hence $c\neq 0$ and there exists $\tilde c \neq 0$ such that 
\begin{equation}
\label{eq:suppnu}
\supp \nu_{12} \subseteq \{M:\det(M) =0 \}\cup \{M:\det(M) = \tilde c\}.
\end{equation}

\textbf{Step 9.} In this step we prove that in fact $\supp \nu_{12}\subseteq \{\det = \tilde c\}$. First, note that $\supp \nu \subseteq K_g$ by Theorem \ref{thm:props}\ref{it:support},  Lemma \ref{lem:dist} and our hypotheses \eqref{eq:hyp}. Thus,  by \eqref{eq:suppnu}, we have
\begin{equation}
\label{eq:qrmeasure}
\supp \nu_{34} \subseteq \{0\}\cup \{g'(\tilde c) M: M\in \R^{2\times 2}, \det M = \tilde c\}.
\end{equation}
Since $\nu_{34}$ has bounded support, \eqref{eq:qrmeasure} implies that $\nu_{34}$ is a \textit{quasiregular gradient Young measure}, namely it is supported on $\{M: |M|^2 \leq K \det(M)\}$ for some $K$, up to changing orientations so that $\tilde c > 0$.  The 0-1 law for such measures \cite[Theorem 1.3]{Astala2002} asserts that either $\nu_{34}=\delta_0$ or else $\nu_{34}(\{0\})=0$.  The former case cannot happen, as otherwise by Theorem \ref{thm:props}\ref{it:strong} we would get $\D v_j\to 0$ in $L^1$.  In turn, this would imply $B = 0$, which is in contradiction with $\eqref{eq:c2j}$, since $c,\det A \neq 0$. Hence we instead have $\nu_{34}(\{0\})=0$ and so,  by \eqref{eq:qrmeasure}, it must be that $\supp \nu_{34} \subseteq \{g'(\tilde c) M: \det M = \tilde c\}$. This implies that  $\nu_{12}$ is supported in $\{\det = \tilde c\}$, as desired.

\textbf{Step 10.} By the previous step and \eqref{eq:minors},  we have
$$\tilde c=\langle \nu_{12}, \det \rangle = \det \overline \nu_{12} = \det A,$$
hence the claimed support condition on $\nu$ holds.  For the final statements note that the previous step shows that $\delta_{\det A}$ is the Young measure generated by the sequence $(\det \D u_j)$, and so Theorem \ref{thm:props}\ref{it:strong} shows that $\det \D u_j \to \det A$ in $L^1(\Omega)$.  This  strong convergence together with the assumption $g'(\det \D u_j) \D u_j - \D v_j\to 0$ in $L^1$ imply that the three sequences $(g'(\det A) \D u_j)$, $(g'(\det \D u_j) \D u_j)$, $(\D v_j)$ have the same weak limit, hence $g'(\det A) A = B$, i.e.\ $(A,B)\in K_g$.  
\end{proof}

\begin{proof}[Proof of Theorem  \ref{thm:cpt}]
Since $(u_j)\subset W^{1,\infty}$ is bounded and $\Omega$ is also bounded,  it suffices to show that $\det \D u_j\to \det \D u$ in $L^1_\tp{loc}(\Omega)$, as this implies the claimed $L^p$-convergence, which in turn implies that $u$ solves \eqref{eq:EL}.  Thus we can without loss of generality assume that $\Omega$ is a ball.  

If $(u_j)$ satisfies \eqref{eq:approxEL} then, according to Lemma \ref{lem:dist}, there is a bounded sequence $(v_j)\subset W^{1,\infty}(\Omega,\R^2)$ such that $w_j=(u_j,v_j)$ satisfies $\dist(\D w_j,K_g)\to 0$ in $L^1(\Omega)$. We want to show that every subsequence of $(u_j)$ admits a further subsequence whose Jacobians converge strongly and thus, up to passing to non-relabeled subsequences, we can assume that $(\D w_j)$ generates a gradient Young measure $(\nu_x)_{x\in \Omega}$ according to Theorem \ref{thm:props}\ref{it:cpt}.  Let us write $\mu_x\equiv (\pi_{12})_\#\nu_x$. By Theorem \ref{thm:props}\ref{it:support},\ref{it:loc} for a.e.\ $x$ we have $\nu_x\in \mc M^\tp{qc}(K_g)$ and by Theorem \ref{thm:props}\ref{it:bc} and Lemma \ref{lem:dist} we see that,  for a.e.\ $x$,  $\nu_x$ can be generated by a bounded sequence $(\D \tilde w_j)=(\D \tilde u_j, \D \tilde v_j)\subset W^{1,\infty}(\Omega,\R^4)$ satisfying \eqref{eq:hyp}.  Thus, applying Theorem \ref{thm:mainapprox}, we see that there is $c\colon \Omega\to \R$ such that $\mu_x$ is supported in $\{M\in \R^{2\times 2}: \det M = c(x)\}$ for a.e.\ $x\in \Omega$.  In light of \eqref{eq:minors} and the weak continuity of Jacobians $\det \D u_j\wstar \det \D u$ in $L^\infty(\Omega)$, we must have
$$c(x) = \det(\overline{\mu_x}) = \langle \mu_x, \det\rangle = \det \D u(x)\quad \tp{for a.e.\ } x\in \Omega.$$
Hence the Young measure generated by $(\det \D u_j)$ is $(\delta_{\det(\D u(x))})_{x\in \Omega}.$ By Theorem \ref{thm:props}\ref{it:strong}, this implies the desired strong convergence.
\end{proof}

\begin{proof}[Proof of Corollary \ref{cor:qc}]
This follows from \eqref{eq:Kqc}, Theorem \ref{thm:props}\ref{it:support},\ref{it:bc}, Lemma \ref{lem:dist} and also Theorem \ref{thm:mainapprox}.
\end{proof}

	\let\oldthebibliography\thebibliography
	\let\endoldthebibliography\endthebibliography
	\renewenvironment{thebibliography}[1]{
	\begin{oldthebibliography}{#1}
	\setlength{\itemsep}{0.5pt}
	\setlength{\parskip}{0.5pt}
	}
	{
	\end{oldthebibliography}
	}
	
	{\small
	\bibliographystyle{abbrv-andre}
	\bibliography{library}
	}

\end{document}